\documentclass[9pt,shortpaper,twoside,web]{ieeecolor}

\usepackage{generic}
\usepackage{cite}
\usepackage{amsmath,amssymb,amsfonts}
\usepackage{graphicx}
\usepackage{textcomp}
\usepackage{xcolor}
\usepackage{algorithm}
\usepackage{algorithmic}
\usepackage{url}

\usepackage{tikz}
\usetikzlibrary{patterns}
\usetikzlibrary{external}
\newtheorem{theorem}{Theorem}

\newtheorem{definition}{Definition}
\newtheorem{lemma}{Lemma}
\newtheorem{remark}{Remark}

\usepackage{caption}
\usepackage{subcaption}

\def\BibTeX{{\rm B\kern-.05em{\sc i\kern-.025em b}\kern-.08em
    T\kern-.1667em\lower.7ex\hbox{E}\kern-.125emX}}
\begin{document}
\title{Fourier--Hermite Dynamic Programming for Optimal Control}

\author{Syeda Sakira Hassan, Simo Särkkä \IEEEmembership{Senior Member, IEEE}
}

\maketitle

\begin{abstract}
In this paper, we propose a novel computational method for solving non-linear optimal control problems. The method is based on the use of Fourier--Hermite series for approximating the action-value function arising in dynamic programming instead of the conventional Taylor-series expansion used in differential dynamic programming (DDP). The coefficients of the Fourier--Hermite series can be numerically computed by using sigma-point methods, which leads to a novel class of sigma-point based dynamic programming methods. We also prove the quadratic convergence of the method and experimentally test its performance against other methods.
\end{abstract}


\section{Introduction}
\label{sec:introduction}
\IEEEPARstart{T}{rajectory} optimization in nonlinear systems is an active research area in optimal control and reinforcement learning~\cite{sutton2018reinforcement, bertsekas2000dynamic, bertsekas2011dynamic}. The aim is to find a state-control sequence that globally or locally minimizes a given performance index such as a cost or a reward function. Applications include trajectory planning in autonomous vehicles, robotics, industrial automation, and gaming \cite{lavalle2006planning, camacho2013model, abbeel2007application, biggs2009optimal, zhao2020pareto, busoniu2017reinforcement}. 

A commonly used approach for solving trajectory optimization problems is dynamic programming (DP) \cite{bellman1957dynamic,bertsekas2000dynamic} which is based on solving the value function from the Bellman's equation~\cite{bellman1957dynamic} by using suitable numerical methods. One such particular approach is differential dynamic programming (DDP) \cite{mayne1966second,jacobson1970differential,mayne1973differential}, where a locally optimal solution is reached iteratively by backward and forward passes. The method is based on the second-order Taylor series expansion of the action-value function that appears in the Bellman's equation of dynamic programming. The convergence of DDP has also been proven under suitable differentiability conditions~\cite{murray1984differential,liao1990proof,liao1991convergence}.

Although DDP has turned out to be useful in many applications, the second-order Taylor series expansion used in this method is computationally expensive due to the higher order derivatives appearing in the expansion. Therefore, researchers have opted to discard the second-order derivatives which has led to methods such as the iterative linear quadratic regulator (iLQR) \cite{li2004iterative}. Furthermore, Taylor series expansion is also an inherently local approximation as it is based on derivatives evaluated at a single point and it induces strong differentiability assumptions on the dynamic and cost functions~\cite{murray1984differential,liao1990proof,liao1991convergence}. To address these limitations, the Taylor series expansion can also be replaced with other approximations.  Examples of such methods are the unscented DP \cite{manchester2016derivative}, sparse Gauss--Hermite quadrature DDP \cite{he2019computational}, and sampled DDP \cite{rajamaki2016sampled}.

In particular, unscented DP \cite{manchester2016derivative} uses an unscented transform based method, inspired by the unscented Kalman filter \cite{Julier+Uhlmann+Durrant-Whyte:2000,sarkka2013bayesian}, to estimate the derivatives using a sigma-point scheme. This allows the DP algorithm to be derivative-free while leveraging information beyond a single point of evaluation and without compromising the performance of the classical DDP algorithm. Additionally, cubature approximations of stochastic continuous-time DDP are considered in \cite{tassa2011high} and probabilistic approximations based on Gaussian processes are considered in~\cite{pan2014probabilistic}.

The contribution of this paper is to propose a method based on Fourier--Hermite series (cf.\ \cite{sarmavuori2011fourier}) to approximate the action-value function. The resulting Fourier--Hermite dynamic programming (FHDP) algorithm can be implemented using sigma-point methods in a completely derivative-free manner, which leads to a new class of sigma-point dynamic programming (SPDP) methods. We also prove the local second-order convergence of the method and experimentally evaluate its performance against classical DDP and unscented DP. Unlike unscented DP or sparse Gauss--Hermite DDP, the method is guaranteed to converge in well-defined conditions, and it can also explicitly handle non-quadratic costs. Moreover, unscented DP requires the propagation of estimates in backward direction along the trajectory, which is not needed in our method. 

The paper is structured as follows. In Sec.~\ref{sec:background}, we revisit the DDP in discrete-time domain. In Sec.~\ref{sec:fh}, we first discuss Fourier--Hermite series and then use the Fourier--Hermite expansion to approximate the action-value function, leading to the proposed method. In Sec.~\ref{sec:theory}, we analyze the computational complexity and prove the local convergence of the method, and in Sec.~\ref{sec:experiments} we experimentally evaluate its performance. Concluding remarks are given in Sec.~\ref{sec:conclusion}.

\section{Differential Dynamic Programming}
\label{sec:background}
In this section, we define the control problem to be solved and review the differential dynamic programming (DDP) algorithm. 

\subsection{Problem formulation}
\label{ssec:problem_formulation}

Consider a nonlinear discrete-time deterministic optimal control problem~\cite{bertsekas2000dynamic} with cost
\begin{equation}
    J(u_{1:T-1}; x_1) =  \ell_T(x_T) + \sum_{k=1}^{T-1} \ell_k(x_k, u_k)
    \label{eq:total_cost}
    \end{equation}
with given initial state $x_1$, subject to the dynamics of the form
\begin{equation}
    x_{k+1} = f_k(x_k, u_k), \quad k=1, \cdots, T-1.
\label{eq:dynamics}
\end{equation}
Here, $x_k \in \mathbb{R}^{n}$ is the state variable, $u_k \in \mathbb{R}^{s}$ is the control variable at step $k$, and $u_{1:T-1} = \{u_1, \ldots, u_{T-1}\}$ is a sequence of controls over the horizon $T$. For a given initial state $x_1$, the total cost of the control sequence $u_{1:T-1}$ is given by~\eqref{eq:total_cost}. Furthermore, $\ell_T(x_T)$ denotes the terminal cost of the state $x_T$ and the $\ell_k(x_k, u_k)$ is the cost incurred at time step $k$.

The aim is to find a control sequence $u^*_{1:T-1}$ that minimizes the cost defined by~\eqref{eq:total_cost}:
\begin{equation}
  u^*_{1:T-1} = \arg \min_{u_{1:T-1}} J(u_{1:T-1}; \, x_1).
\end{equation}
This solution can be expressed in terms of the optimal cost-to-go or value function $V_k(x_k)$ that gives the minimum total cost accumulated between time step $k$ and $T$ starting from the state $x_k$. As shown by Bellman~\cite{bellman1957dynamic}, we can compute the value function using backward recursion as follows:
\begin{equation}
    \begin{split}
    V_k(x_k) &= \min_{u_k} \left\{ \ell_k(x_k, u_k) + V_{k+1}(f_k(x_{k}, u_k)) \right\},
    \end{split}
    \label{eq:bellman}
\end{equation}
where the value function at the terminal time $T$ is $V_{T}(x_T) = \ell_T(x_T)$.

Because \eqref{eq:bellman} becomes computationally infeasible with increasing state-dimensionality \cite{bellman1959functional}, one common approach is to approximate the action-value function appearing on the right hand side of \eqref{eq:bellman}:
\begin{equation}
Q_k(x_k, u_k) = \ell_k(x_k, u_k) + V_{k+1}(f_k(x_{k}, u_k)),
\label{eq:q_function}
\end{equation}
using a tractable approximation. In particular, differential dynamic programming (DDP), which is discussed below, uses a second order Taylor series expansion for this purpose.

\subsection{Differential dynamic programming}
\label{ssec:deterministic_ddp}
The classical DDP~\cite{mayne1966second, jacobson1970differential, mayne1973differential} approach uses a second-order Taylor series expansion of the action-value function $Q_k$ about a nominal trajectory. Given a nominal trajectory of states and controls $(\hat{x}_k, \hat{u}_k)$, at step $k$, we can approximate $Q_k$ around this trajectory using a second-order Taylor series expansion 
\begin{equation}
\begin{split}
    Q_k( x_k, u_k) &\approx Q_k^0 + Q_x^\top \, \delta x_k + Q_u^\top \, \delta u_k \\
    \quad& + \frac{1}{2} \begin{bmatrix}
      \delta x_k^\top &   \delta u_k^\top
    \end{bmatrix}  \begin{bmatrix}
       Q_{xx} & Q_{xu}  \\
      Q_{ux} & Q_{uu}
    \end{bmatrix} \begin{bmatrix}
      \delta x_k \\
      \delta u_k
    \end{bmatrix},
\end{split}
\label{eq:approx_q_func}
\end{equation}
where $\delta x_k = x_k - \hat{x}_k$ and $\delta u_k = u_k - \hat{u}_k$. Let us now assume a quadratic approximation for the value function of the form
\begin{equation}
    \begin{split}
  V_{k+1}(f_{k}(x_k, u_k)) &\approx V_{k+1}^0 -  v_{k+1}^\top \, \delta x_{k+1} +  \frac{1}{2} \, \delta x_{k+1}^\top \, S_{k+1} \, \delta x_{k+1}.
  \end{split}
    \label{eq:approx_value_function}
\end{equation}
If we now form a second-order Taylor series expansion of \eqref{eq:q_function}, we get the following coefficients for \eqref{eq:approx_q_func}:
\begin{equation}
    \begin{split}
    Q_k^0 &\approx \ell_k(\hat{x}_k, \hat{u}_k)  + V_{k+1}^0 - v_{k+1}^\top \, (f_k(\hat{x}_k, \hat{u}_k) - \hat{x}_{k+1}) \\
    &+ \frac{1}{2} (f_k(\hat{x}_k, \hat{u}_k) - \hat{x}_{k+1})^\top \, S_{k+1} \, (f_k(\hat{x}_k, \hat{u}_k) - \hat{x}_{k+1}), \\
    Q_x &= L_x + F_x^\top [- v_{k+1} + S_{k+1}\, (f_k(\hat{x}_k, \hat{u}_k ) - \hat{x}_{k+1})],  \\
    Q_u &= L_u + F_u^\top [- v_{k+1} + S_{k+1} \,(f_k(\hat{x}_k, \hat{u}_k ) - \hat{x}_{k+1})],  \\
  Q_{xx} &= L_{xx} + F_x^\top \, S_{k+1} \, F_x, \\
  &\quad+ \sum_m F^m_{xx} [- v_{k+1} + S_{k+1} \, (f_k(\hat{x}_k, \hat{u}_k ) - \hat{x}_{k+1})]_m,  \\
  Q_{xu} &= L_{xu} + F_x^\top\, S_{k+1} \, F_u, \\
  &\quad+ \sum_m F^m_{xu} [- v_{k+1} + S_{k+1} \, (f_k(\hat{x}_k, \hat{u}_k ) - \hat{x}_{k+1})]_m, \\
  Q_{uu} &= L_{uu} + F_u^\top \, S_{k+1} \, F_u, \\
  &\quad+ \sum_m F^m_{uu} [- v_{k+1} + S_{k+1} \, (f_k(\hat{x}_k, \hat{u}_k ) - \hat{x}_{k+1})]_m. \\
  \label{eq:coefficients_of_q}
    \end{split}
\end{equation}
Above we have denoted the gradients of $\ell_k$ with respect to $x$ and $u$ as $L_x$ and $L_u$. The second-order derivative matrices of $\ell_k$ are denoted as $L_{xx}$, $L_{xu}$, and $L_{uu}$. The Jacobians of $f_k$ with respect to $x$ and $u$ are denoted as $F_x$ and $F_u$. In addition, we use $F_{xx}^m$, $F_{xu}^m$, and $F_{uu}^m$ to denote the second-order derivative matrices of the $m^{th}$ component of $f_k$. All the derivatives are evaluated at $(\hat{x}_k, \hat{u}_k)$.

Minimizing \eqref{eq:approx_q_func} with respect to $\delta u_k$, we arrive at the following correction to the control trajectory
\begin{equation}
    \begin{split}
        \delta u_k &= - \, Q_{uu}^{-1} \, Q_u  - Q_{uu}^{-1} \, Q_{ux} \, \delta x_k.\\
    \end{split}
\label{eq:optimal_control_law}
\end{equation}
Let us define
\begin{equation}
  \begin{split}
  d &= - Q_{uu}^{-1} Q_u, \quad K =  Q_{uu}^{-1} Q_{ux},
  \end{split}
\label{eq:d_and_K}
\end{equation}
then we can rewrite \eqref{eq:optimal_control_law} as follows:
\begin{equation}
    \begin{split}
        \delta u_k &= d - K  \delta x_k.
    \end{split}
\label{eq:optimal_control_law_modified}
\end{equation}
By substituting $\delta u_k$ to~\eqref{eq:approx_q_func} we get the coefficients for the quadratic approximation of the value function at step $k$:
\begin{equation}
    \begin{split}
        V_k^0 &= Q_k^0 + \frac{1}{2} d^\top Q_{u}, \\
        v_k &= - \, Q_x - K^\top Q_{uu} \, d  ,  \\
        S_k &= Q_{xx} - K^\top Q_{uu} \, K.
    \end{split}
    \label{eq:value_coefficient_and_control_variation}
\end{equation}
This procedure is then continued backwards for $k-1,k-2,\ldots,1$. 

That is, the backward pass of DDP starts from the terminal time step $k = T$ from a quadratic approximation to $\ell_T$ formed with a second-order Taylor series expansion centered at $\hat{x}_T$. Then, we successively perform the aforementioned computations until $k=1$. The backward pass is followed by a forward pass, where the system is simulated forward in time under the optimal control law \eqref{eq:optimal_control_law_modified} to generate a new trajectory. The backward and forward passes are iterated until convergence. The pseudocode for the classical DDP method is given in Algorithm~\ref{alg:ddp_algorithm}.

\begin{algorithm}[tbh]
\caption{Differential dynamic programming}
\label{alg:ddp_algorithm}
\begin{algorithmic}[1]
\renewcommand{\algorithmicrequire}{\textbf{Input:}}
\renewcommand{\algorithmicensure}{\textbf{Output:}}
\REQUIRE Initial state $\hat{x}_1$, nominal control $\hat{u}_k$ for $k=1,\ldots,T-1$, and nominal states $\hat{x}_k$ for $k=2,\ldots,T$
\ENSURE Updated $\hat{u}_k$ and $\hat{x}_k$
\STATE \textbf{Backward pass:}
\STATE {Compute terminal cost, $\ell_T(\hat{x}_T)$ and its derivatives $L_{x}(\hat{x}_T)$ and $L_{xx}(\hat{x}_T)$}
\STATE {$V_T^0 \gets \ell_T(\hat{x}_T) $, $v_T \gets -L_x(\hat{x}_T)$, and $S_{T} \gets L_{xx}(\hat{x}_T) $}
\FOR {$k = T-1$ to $1$}
\STATE {Evaluate the partial derivatives $L_x$, $L_u$, $L_{xx}$, $L_{uu}$, $L_{xu}$ of $\ell_k$ and $F_x$, $F_u$, $F_{xx}$, $F_{uu}$, $F_{xu}$ of $f_k$} at $(\hat{x}_k,\hat{u}_k)$.
\STATE {Compute the coefficients of $Q_k(x_k, u_k)$ using~\eqref{eq:coefficients_of_q}.}
\STATE {Compute $d$ and $K$ using \eqref{eq:d_and_K}, and $v_k$ and $S_k$ using~\eqref{eq:value_coefficient_and_control_variation}.} 
\ENDFOR
\STATE \textbf{Forward pass: }
\STATE {Start from $\hat{x}_1$}
\FOR{$k = 1$ to $T-1$}
\STATE {$\hat{u}_k \gets u_k + \delta u_k$, where $\delta u_k$ is given by \eqref{eq:optimal_control_law_modified}.}
\STATE {$\hat{x}_{k+1} \gets f_k(\hat{x}_k, \hat{u}_k)$}
\ENDFOR
\STATE Repeat from Step 1 until convergence. 
\end{algorithmic}
\end{algorithm}

\subsection{Regularization of the optimization and line search} 
\label{sec:reg}
As with all nonlinear optimization, proper care must be taken to ensure a good convergence behavior of the method. The DDP algorithm involves the matrix inversion of $Q_{uu}$ in~\eqref{eq:optimal_control_law}, which may cause numerical instability. A regularization scheme was, therefore, proposed by~\cite{jacobson1970differential,liao1991convergence,tassa2011theory} to ensure invertibility of $Q_{uu}$ in~\eqref{eq:optimal_control_law} by replacing it with: 
\begin{equation}
    \begin{split}
        \tilde{Q}_{uu} &= Q_{uu} + \beta I,
    \end{split}
    \label{eq:regularized_control_coef}
\end{equation}
where, $\beta > 0$ is a small positive constant. Furthermore, when using this regularization, \cite{tassa2011theory} suggests that the following modifications to \eqref{eq:value_coefficient_and_control_variation} are recommended for numerical stability:
\begin{equation}
    \begin{split}
        V_k^0 &= Q_k^0 + \frac{1}{2} \, \tilde{d}^\top \, Q_{u} + \frac{1}{2} \, \tilde{d}^\top \, Q_{uu} \tilde{d},  \\
        v_k &= - \, Q_x + \tilde{K}^\top \, Q_{uu} \, \tilde{d}  + \tilde{K}^\top \, Q_{u}  - Q_{ux}^\top \, \tilde{d},  \\
        S_k &= Q_{xx} + \tilde{K}^\top Q_{uu} \, \tilde{K} - \tilde{K}^\top \, Q_{ux} - Q_{ux}^\top \, \tilde{K},
    \end{split}
    \label{eq:modified_value_coefficient_and_control_variation}
\end{equation}
where $\tilde{K} = \tilde{Q}_{uu}^{-1} Q_{ux}$ and $\tilde{d} = - \, \tilde{Q}_{uu}^{-1} Q_{u}$.

The value $\beta$ can be adapted by using a Levenberg--Marquardt type of adaptation procedure, that is, if the cost of the new trajectory is less than the current one, then the value of $\beta$ is decreased by dividing it with a constant factor $\nu > 1$, otherwise the value increased by multiplying with $\nu$ and the new trajectory is discarded. Note that if $\beta = 0$, then~\eqref{eq:modified_value_coefficient_and_control_variation} reduces to \eqref{eq:value_coefficient_and_control_variation}.  In~\cite{tassa2012synthesis} the authors used a quadratic modification schedule to choose $\beta$ at each iteration. 

References \cite{yakowitz1984computational,liao1990proof,tassa2011theory} also suggest to use an additional backtracking line search scheme to improve convergence. A parameter $0 < \epsilon \leq 1$ is introduced in the update statement of the control in the forward pass routine as follows
\begin{equation}
\begin{split}
    \delta \hat{u}_k = \epsilon \, \tilde{d} - \tilde{K}\,\delta x_k, \quad \hat{u}_k = u_k + \delta \hat{u}_k.
\end{split}
\label{eq:forward_pass_backtracking}
\end{equation}
We start by setting $\epsilon = 1$ and update control using~\eqref{eq:forward_pass_backtracking} and then generate a new state sequence by forward simulation, that is, $\hat{x}_{k+1} = f(\hat{x}_k, \hat{u}_k)$. If the decrease in the cost function is not below a given threshold, the value of $\epsilon$ is decreased (e.g. by halving it as we did in our case), and we restart the forward pass again. 

\subsection{Implementation of DDP using automatic differentiation}
\label{sec:autodiff}
During the backward pass in DDP, we need to evaluate the derivatives on the right-hand side of~\eqref{eq:coefficients_of_q} at $(\hat{x}_k, \hat{u}_k)$ to approximate $Q_k$. Classically these derivatives have been derived by hand or via symbolic or numerical differentiation methods, but they can also be automatically computed by using automatic differentiation (AD) \cite{griewank2008evaluating}. AD is based on transforming the function to be evaluated into a sequence of operations that compute the exact derivatives of the function along with its value. AD is readily available in several programming platforms such as TensorFlow~\cite{abadi2016tensorflow}, PyTorch~\cite{paszke2017automatic}, and MATLAB~\cite{matlabautodiff}.

When using automatic differentiation, there are two alternative ways to evaluate $Q_k$ of the DDP algorithm at the nominal trajectory $(\hat{x}_k, \hat{u}_k)$. The first approach evaluates the derivatives of $\ell_k$ and $f_k$ on the right-hand side of~\eqref{eq:coefficients_of_q} at $(\hat{x}_k, \hat{u}_k)$ using AD. Once we have all the derivatives,  we can solve for $Q_k$ using~\eqref{eq:coefficients_of_q}. The other alternative is to use AD directly to $Q_k$ and evaluate its derivatives at $(\hat{x}_k, \hat{u}_k)$. In this paper, we implement DDP with AD by applying the former approach due to its more direct connection with the classical DDP.

\section{Fourier--Hermite Dynamic Programming} 
\label{sec:fh}
In this section, we present the proposed method which is based on using the Fourier--Hermite series approximation instead of the Taylor series approximation for the action-value function.

\subsection{Fourier--Hermite series}

Fourier--Hermite series is a series expansion of a function using Hermite polynomial basis of a Hilbert space~\cite{malliavin2015stochastic}. The $m$-th order univariate Hermite polynomials can be computed as follows:
\begin{equation}
    H_m(x) = (-1)^m \exp{(x^2/2)} \frac{d^m}{d x^m} \exp{(-x^2/2)}, \quad m = 0, 1, \ldots.
\end{equation}
The first few ($m=\{0, 1, 2\}$) Hermite polynomials are $H_0(x) = 1$, $H_1(x) = x$, $H_2(x) = x^2 - 1$, and for $m\geq3$ polynomials can be found using the recursion $H_{m+1}(x) = x \, H_{m}(x) - m \, H_{m-1}(x)$.

A multivariate Hermite polynomial with multi-index $\mathcal{I} = \{i_1, \ldots, i_n\}$ for $n$-dimensional vector $x$ can be written as
\begin{equation}
    H_\mathcal{I}(x) = \prod_{m=1}^n H_{i_m}(x_m),
    \label{eq:mth_hermite}
\end{equation}
where $H_{i_m}(x_m)$ are univariate Hermite polynomials. Let us define the inner product of two functions $f$ and $g$ as
\begin{equation}
\begin{split}
\langle f, g \rangle &=  \int_{\mathbb{R}^n} f(x) \, g(x) \, \mathcal{N}(x\mid 0, I) \, d x,
\end{split}
\end{equation}
and a Hilbert space $\mathcal{H}$ consisting of functions satisfying $\| g \|^2 = \langle g, g \rangle < \infty$. Then the Hermite polynomials are orthogonal in the sense
\begin{equation}
\begin{split}
\langle H_\mathcal{I}, H_\mathcal{J} \rangle &=  \int H_\mathcal{I}(x) \, H_\mathcal{J}(x)  \, \mathcal{N}(x\mid 0, I) \, dx
=   \begin{cases}
        \mathcal{I}!, & \text{if } \mathcal{I} = \mathcal{J},  \\    
        0, & \text{otherwise}.    
    \end{cases}
\end{split}
\label{eq:multivariate_orthogonality}
\end{equation}
Here, $\mathcal{I}! = i_1! \, i_2! \, \cdots \, i_n!$, 
$\mathcal{J} = \{ j_1, j_2,\cdots, j_n \}$ and $\mathcal{I} = \mathcal{J}$ when $i_k = j_k$ for all elements in $\mathcal{I}, \mathcal{J}$. Now, we can define the Fourier--Hermite expansion of a function $g(x)$ as follows
\begin{definition}
For any $g \in \mathcal{H}$, the Fourier--Hermite expansion of $g$ with respect to a unit Gaussian distribution $\mathcal{N}(0,I)$ is given by
\begin{equation}
    g(x) = \sum_{k=0}^\infty \sum_{|\mathcal{I}|=k}
    \frac{1}{\mathcal{I}!} \, {c}_{\mathcal{I}} \, H_{\mathcal{I}}(x),
\label{eq:fourier_hermite_series}
\end{equation}
where, $H_{\mathcal{I}}$ is a multivariate Hermite polynomial and ${c}_{\mathcal{I}}$ are the series coefficients given by the inner product ${c}_{\mathcal{I}} = \langle g, H_{\mathcal{I}} \rangle$.
\end{definition}

The representation in~\eqref{eq:fourier_hermite_series} is useful if we want to compute expectations of a nonlinear function over a unit Gaussian distribution. It turns out that the expectation of the function can be simply extracted from the zeroth order coefficient $c_0$ of the Fourier--Hermite series and the higher order coefficients are equal to the expectations of the derivatives of the function $g(x)$ \cite{sarmavuori2011fourier}. In this paper, we are particularly interested in the second-order Fourier--Hermite series expansion of $g(x)$ which can be written as
\begin{equation}
\begin{split}
g(x) &\approx \sum_{k=0}^2  \sum_{|\mathcal{I}|=k} \frac{1}{\mathcal{I}!} \, {c}_{\mathcal{I}} \, H_{\mathcal{I}}(x)
= \mathbb{E} \, [g(x)] + \mathbb{E} \, [g(x) \, H_1(x)]^\top \, H_1(x) \\
&\qquad + \frac{1}{2} \, \mathop{\text{tr}}\left\{\mathbb{E} \,[g(x) \, H_2(x)] \, H_2(x)\right\}.
\end{split}
\label{eq:fh_second_order}
\end{equation}
In \eqref{eq:fh_second_order}, the multivariate polynomials $H_i(x)$ have been expressed as vectors and matrices as follows (cf. \cite{sarmavuori2011fourier}): 
\begin{equation}
    \begin{split}
        H_0(x) = 1,\quad
        H_1(x) = x,\quad
        H_2(x) = x \, x^\top - I.
    \end{split}
    \label{eq:tensors}
\end{equation}
We can also generalize the expansion to a more general Gaussian distribution $\mathcal{N}(\mu, \Sigma)$ by rewriting the second order Fourier--Hermite expansion as
\begin{equation}
\begin{split}
g(\Lambda \, x + \mu) &\approx \mathbb{E} \, [g(\Lambda \, x + \mu)] + \mathbb{E} \, [g(\Lambda \, x + \mu) \, H_1(x)]^\top \, H_1(x) \\
&\quad + \frac{1}{2} \, \mathop{\text{tr}} \, \left\{\mathbb{E} \,[g(\Lambda \, x + \mu) \, H_2(x)] \, H_2(x)\right\}.
\end{split}
\label{eq:fh_general_second_order}
\end{equation}
Above, we have put $\Sigma = \Lambda \, \Lambda^\top$. If we now let $y = \Lambda \, x + \mu$, then \eqref{eq:fh_general_second_order} becomes
\begin{equation}
\begin{split}
g(y) &\approx \mathbb{E} \, [g(y)] + \mathbb{E} \, [g(y) \, H_1 \,(\Lambda^{-1} \, (y - \mu))]^\top \, H_1 \,(\Lambda^{-1} \, (y - \mu)) \\
&\quad + \frac{1}{2} \, \mathop{\text{tr}} \, \left\{\mathbb{E} \, [g(y) \, H_2 \, (\Lambda^{-1} \, (y - \mu))] \, H_2 \,(\Lambda^{-1} \, (y - \mu))\right\},
\end{split}
\label{eq:fh_general_second_order_y}
\end{equation}
where the expectations are over $y \sim \mathcal{N}(\mu,\Sigma)$. Now, if we substitute the multivariate Hermite polynomials from~\eqref{eq:tensors} to~\eqref{eq:fh_general_second_order_y}, we get
\begin{equation}
\begin{split}
g(y) &\approx \mathbb{E} \, [g(y)] + \mathbb{E} \, [g(y) \, H_1 \, (\Lambda^{-1} \, (y - \mu))]^\top \, (\Lambda^{-1} \, (y - \mu)) \\
&\quad + \frac{1}{2} \mathop{\text{tr}}\left\{\mathbb{E} \, [g(y) \, H_2 \, (\Lambda^{-1} \, (y - \mu))] \right.\\
&\qquad \times \left.
(\Lambda^{-1}(y -\mu) \, (y - \mu)^\top \, \Lambda^{-\top} - I) \right\}.
\end{split}
\label{eq:fh_general_second_order_y2}
\end{equation}
Let us denote
\begin{equation}
    \begin{split}
        a_G &= \mathbb{E} \, [g(y)],\\
        b_G &= \mathbb{E} \, [g(y) \, H_1 \, (\Lambda^{-1} \, (y - \mu))],\\
        C_G &= \mathbb{E} \, [g(y) \, H_2 \, (\Lambda^{-1} \, (y - \mu))].
    \end{split}
    \label{eq:coef_def}
\end{equation}
Then, \eqref{eq:fh_general_second_order_y2} can be rewritten as
\begin{equation}
\begin{split}
g(y) &\approx a_G - \frac{1}{2} \, \mathop{\text{tr}} \, \{C_G\} + b_G^\top \, (\Lambda^{-1} \,(y - \mu))\\
&\quad + \frac{1}{2} \, (y - \mu)^\top \, [\Lambda \, C_G^{-1} \, \Lambda^{\top}]^{-1} \,(y - \mu).
\end{split}
\label{eq:fh_general_second_order_y3}
\end{equation}
Now, we are ready to apply this approximation to the action-value function.
\subsection{Fourier--Hermite approximation of action-value function}
Consider the approximation of the action-value function $Q_k$ defined in~\eqref{eq:q_function}. Furthermore, assume that we have a quadratic approximation to the value function at step $k+1$ of the form \eqref{eq:approx_value_function}. Instead of using a Taylor series approximation to get \eqref{eq:approx_q_func} as in \eqref{eq:coefficients_of_q}, we will now form the approximation with Fourier--Hermite series. Assume that our nominal trajectory for $k=1,\ldots,T-1$ consists of mean $\mu_k = [\hat{x}_k ,\, \hat{u}_k]$ and the joint covariance $\Sigma_k= \Lambda_k \, \Lambda_k^\top$ for the Gaussian distribution of the state-control pair $(x_k, u_k)$\footnote{With a slight abuse of the notation $[x_k ,\,  u_k]$ used here to represent a column vector and is equivalent to $\begin{bmatrix}
   x_k \\ u_k
\end{bmatrix}$, not to be confused with $[ x_k^\top \, \,  u_k^\top]$, which represents a row vector.}. Further assume that at the terminal step $T$ the nominal trajectory consists of $\hat{x}_T$ and $\Sigma_T$. If we let $\delta x_k = x_k - \hat{x}_k$ and $\delta u_k = u_k - \hat{u}_k$, then the Fourier--Hermite approximation for $Q_k$ can be written as
\begin{equation}
    \begin{split}
        Q_k(x_k, u_k) &\approx a_Q - \frac{1}{2} \, \mathop{\text{tr}} \{C_Q\} + b_Q^\top \, \Lambda_k^{-1} \begin{bmatrix} 
        \delta x_k \\ \delta u_k
        \end{bmatrix} \\
        &\quad +\frac{1}{2} \, \begin{bmatrix} 
        \delta x_k^\top  & \delta u_k^\top
        \end{bmatrix} \left[\Lambda_k \, C_Q^{-1} \, \Lambda_k\right]^{-1} \begin{bmatrix} 
        \delta x_k \\ \delta u_k
        \end{bmatrix},
    \end{split}
    \label{eq:fh_approx_q_func}
\end{equation}
where
\begin{equation}
    \begin{split}
        a_Q &= \mathbb{E} \, \left[Q_k(x_k, u_k)\right], \\
        b_Q &= \mathbb{E} \, \left[Q_k(x_k, u_k) \, H_1 \, (\Lambda_k^{-1}\begin{bmatrix}\delta x_k , \delta u_k
        \end{bmatrix}) \right], \\
        C_Q &= \mathbb{E} \, \left[Q_k(x_k, u_k) \, H_2 \, (\Lambda_k^{-1} \, \begin{bmatrix}\delta x_k , \delta u_k
        \end{bmatrix}) \right],
    \end{split}
    \label{eq:aQbQCQ_det}
\end{equation}
with the expectations taken over the joint Gaussian distribution for $(x_k,u_k)$. The expectations can be numerically computed, for example, using numerical integration methods such as sigma-point methods \cite{sarkka2013bayesian}. By matching the terms in \eqref{eq:approx_q_func} and \eqref{eq:fh_approx_q_func} we now notice that this approximation has the same form as DDP with the correspondences
\begin{equation}
    \begin{split}
        Q_k^0 &= a_Q - \frac{1}{2}\mathop{\text{tr}} \{C_Q\},\\
        \begin{bmatrix}
          Q_x^\top \quad Q_u^\top
        \end{bmatrix} &= b_Q^\top \, \Lambda_k^{-1},\\
        \begin{bmatrix}
            Q_{xx} & Q_{xu}\\
            Q_{ux} & Q_{uu}
        \end{bmatrix} &= \left[\Lambda_k \, C_Q^{-1} \, \Lambda_k\right]^{-1}.
    \end{split}
\label{eq:fh_coef}
\end{equation}

At the terminal step $T$, the nominal trajectory consists of mean $\hat{x}_T$ and covariance $\Sigma_T = \Lambda_T \Lambda_T^\top$. The approximation is formed as
\begin{equation}
    \begin{split}
  V_T(x_T) &\approx V_T^0 -  v_T^\top \, \delta x_T  +  \frac{1}{2} \, \delta x_T^\top \, S_T \, \delta x_T,
  \end{split}
\end{equation}
where $\delta x_T = x_T - \hat{x}_T$,  and 
\begin{equation}
\begin{split}
  V_T^0 &= a_T - \frac{1}{2} \mathrm{tr}\left\{ C_T \right\}, \\
      v_T^\top
    &= -\, b_T^\top \, \Lambda_T^{-1}, \\
      S_T
  &= [\Lambda_T \, C_T^{-1} \, \Lambda_T^\top]^{-1},
\end{split}
\label{eq:fh_coef_T}
\end{equation}
with
\begin{equation}
\begin{split}
  a_T &= \mathbb{E}\,[\ell_T(x_T)], \\
  b_T &= \mathbb{E}\,[\ell_T(x_T) \, H_1(\Lambda_T^{-1} \, \delta x_T)], \\
  C_T &= \mathbb{E}\,[\ell_T(x_T) \, H_2(\Lambda_T^{-1} \, \delta x_T)].
\end{split}
\label{eq:aTbTCT_det}
\end{equation}
The Fourier--Hermite dynamic programming (FHDP) algorithm in its abstract form now consists in replacing the Taylor series based computations of the action-value function coefficients in \eqref{eq:approx_q_func} with \eqref{eq:fh_coef} and the terminal step value function approximation by \eqref{eq:fh_coef_T}. It is worth noticing that the Hermite polynomials needed at the terminal step in \eqref{eq:aTbTCT_det} are functions of $n$-dimensional input although in \eqref{eq:aQbQCQ_det} the input dimension is $s + n$.

\subsection{Coefficient computation via sigma-point methods} \label{sec:sigma-points}
Sigma-point methods are numerical integration methods commonly used in non-linear filters, such as unscented Kalman filters (UKFs), cubature Kalman filters (CKFs), Gauss--Hermite Kalman filters (GHKFs), and their extensions \cite{sarkka2013bayesian}. In their most common form, sigma-point method can be seen as Gaussian quadrature approximations for computing Gaussian integrals as follows:
\begin{equation}
\begin{split}
  \int g(x) \, \mathcal{N}(x \mid 0,I) \, dx \approx \sum_{i} W_i^n \, g(\xi_i^n),
\end{split}
\end{equation}
where $x \in \mathbb{R}^n$, and the weights $W_i^n$ and (unit) sigma points $\xi_i^n$ for the $n$-dimensional integration rule are determined by the sigma-point method at hand. By a change of variables, $y = \Lambda \, x + \mu$ we can then approximate integrals over more general Gaussian distributions $\mathcal{N}(y \mid \mu, \Sigma)$ as
\begin{equation}
\begin{split}
  \int g(y) \, \mathcal{N}(y \mid \mu, \Sigma) \, dx \approx \sum_{i} W_i^n \, g(\Lambda \, \xi_i + \mu),
\end{split}
\end{equation}
where, $\Sigma = \Lambda \, \Lambda^\top$.
We can now apply this rule to \eqref{eq:aQbQCQ_det}, which gives the sigma-point approximations: 
\begin{equation}
\begin{split}
  a_Q &\approx \sum_i W_i^{n+s} \, Q_k(\Lambda_k \, \xi_i^{n+s} + [\hat{x}_k ,\, \hat{u}_k]), \\
  b_Q &\approx \sum_i W_i^{n+s} \, Q_k(\Lambda_k \, \xi_i^{n+s} + [\hat{x}_k ,\, \hat{u}_k]) \, \xi_i^{n+s}, \\
  C_Q &\approx \sum_i W_i^{n+s} \, Q_k(\Lambda_k \, \xi_i^{n+s} + [\hat{x}_k ,\, \hat{u}_k])
    \, (\xi_i^{n+s} \, [\xi_i^{n+s}]^\top - I),
\end{split}
\label{eq:aQbQCQ_sddp}
\end{equation}
where $\Sigma_k = \Lambda_k \, \Lambda_k^\top$ is the joint covariance  of the nominal trajectory $(x_k, u_k)$. It is though important to note that it is not sufficient to use a third order rule such as unscented transform or 3rd order cubature rule, because the resulting integrals are typically higher order polynomials. Instead, it is advisable to use, for example, Gauss--Hermite rules~\cite{ito2000gaussian, haykin2009cubature} or higher order spherical cubature (i.e., unscented) rules~\cite{sarkka2015relation, sarkka2014gaussian, wu2006numerical}. 

The expectations at the terminal step \eqref{eq:aTbTCT_det} can be computed as
\begin{equation}
\begin{split}
  a_T &\approx \sum_i W_i^n \, \ell_T(\Lambda_T \, \xi_i^n + \hat{x}_T), \\
  b_T &\approx \sum_i W_i^n \, \ell_T(\Lambda_T \, \xi_i^n + \hat{x}_T) \, \xi_i^n, \\
  C_T &\approx \sum_i W_i^n \, \ell_T(\Lambda_T \, \xi_i^n + \hat{x}_T) \,
    (\xi_i^n \, [\xi_i^n]^\top - I).
\end{split}
\label{eq:aTbTCT_sddp}
\end{equation}

The sigma-point based FHDP is summarized in Algorithm~\ref{alg:spdp_algorithm}. Although the algorithm is written in its simple form, it is also possible to use the regularization and line search methods described in Section~\ref{sec:reg} as part of it. Although in the line search, a straightforward way is to use the original cost function as the merit function, in its implementation, it is important to take into account that the cost function minimized by the FHDP is not exactly the original cost function (see Sec.~\ref{sec:theory}). 

\begin{algorithm}[tbh]
\caption{Sigma-point dynamic programming (SPDP)}
\label{alg:spdp_algorithm}
\begin{algorithmic}[1]
\renewcommand{\algorithmicrequire}{\textbf{Input:}}
\renewcommand{\algorithmicensure}{\textbf{Output:}}
\REQUIRE Initial state $\hat{x}_1$, nominal control $\hat{u}_k$, for $k=1,\ldots,T-1$, nominal state $\hat{x}_k$, for $k=2,\ldots,T$, terminal covariance $\Sigma_{T}$, and joint covariance $\Sigma_{k}$
\ENSURE Update $\hat{u}_k$ and  $\hat{x}_k$
\STATE \textbf{Backward pass:}
\STATE {Given $\Sigma_T$, compute $a_T$, $b_T$, and $C_T$} using~\eqref{eq:aTbTCT_sddp}
\STATE {Compute $V_T$, $v_T$, and $S_T$ using~\eqref{eq:fh_coef_T}}
\FOR {$k = T-1$ to $1$}
\STATE { Given $\Sigma_{k}$, compute $a_Q$, $b_Q$ and $C_Q$ using~\eqref{eq:aQbQCQ_sddp}}.
\STATE {Compute the coefficients of $Q_k(x_k, u_k)$ using~\eqref{eq:fh_coef}}.
\STATE {Compute $d$ and $K$ using \eqref{eq:d_and_K}, and $v_k$ and $S_k$ using~\eqref{eq:value_coefficient_and_control_variation}.}
\ENDFOR
\STATE \textbf{Forward pass: }
\STATE {Start from $\hat{x}_1$.}
\FOR{$k = 1$ to $T-1$}
\STATE {$\hat{u}_k \gets u_k + \delta u_k$, where $\delta u_k$ is given by \eqref{eq:optimal_control_law_modified}.}
\STATE {$\hat{x}_{k+1} \gets f_k(\hat{x}_k, \hat{u}_k)$}
\ENDFOR
\STATE Repeat from Step 1 until convergence. 
\end{algorithmic}
\end{algorithm}

\section{Theoretical Results} \label{sec:theory}

In this section, we discuss the computational complexity of DDP and sigma-point FHDP methods, and prove the local convergence of our proposed method. 

\subsection{Computational complexity}

Let us assume that the dimension of the state $n$ dominates the dimension of the control $s$. When implemented in form \eqref{eq:coefficients_of_q}, the computational complexity of DDP in terms of function evaluations nominally depends on the complexity of evaluating the first-order and the second-order derivatives of the dynamics. Each of the first-order derivatives in~\eqref{eq:approx_q_func} require  $\mathcal{O}(n^2)$ operations. The second-order derivatives require $\mathcal{O}(n^3)$ operations. Therefore, the computational complexity of the DDP method per iteration is $\mathcal{O}(T n^3)$, where $T$ is the time horizon~\cite{liao1991convergence}. However, in some cases, we can decrease the complexity of DDP to $\mathcal{O}(T n^2)$ \cite{nganga2021accelerating} by using automatic differentiation directly on $Q_k$ as described in Section~\ref{sec:autodiff}. 

In sigma-point methods, the computational complexity in terms of number of function evaluations is equal to the number of sigma-points. For instance, in Gauss--Hermite quadrature rule with order $p$, the number of required sigma-points is $p^{m}$, where $m = n + s$. In this rule, the number of  evaluation points grows exponentially with the state and control dimensions. On the other hand, in the 5th order symmetric cubature rule \cite{McNamee+Stenger:1967} (i.e., the 5th order unscented transform), the number of required sigma-points is $2m^2 + 1$ and in the 7th order rule it is $
(4m^3 + 8m + 3)/3$ (see, e.g., \cite{kokkala2015sigma}). Therefore, the total computational complexity per iteration is $O(T m^2)$ when using the 5th order rule and $O(T m^3)$ when using the 7th order rule.

\subsection{Convergence analysis}
\label{sec:convergence}

In this section, we study the local convergence of the Fourier--Hermite dynamic programming method. It is already known that differential dynamic programming (DDP) converges quadratically to the unique minimizer $u^*_{1:T-1}$ in well-defined conditions~\cite{murray1984differential,liao1990proof,liao1991convergence}. These results can be summarized as follows.

\begin{lemma}[DDP convergence]
\label{lem:ddp_convergence}
Assume that $\ell_k$, $k=1, \ldots, T$ and $f_k$, $k=1, \ldots, T-1$ are three times continuously differentiable with respect to $x_k$ and $u_k$, and the second derivative of $Q_k$ with respect to $u_k$ is positive definite for all $k$. Furthermore, assume that the iterates $(x_{1:T}^{(i)},u_{1:T-1}^{(i)})$ produced by DDP are contained in a convex set $D$, which also contains the minimizer $(x_{1:T}^{*},u_{1:T-1}^{*})$. Then, the sequence of DDP iterates $u_{1:T-1}^{(i)}$ converges quadratically in the sense that where exist $c > 0$ such that
\begin{equation}
  \| u_k^{(i+1)} - u^*_k \| \le c \, \| u^{(i)} - u^*_k \|^2.
\end{equation}
\end{lemma}
Our aim is now to prove the convergence of the proposed Fourier--Hermite dynamic programming (FHDP) method by constructing a modified model such that when we apply DDP on it, it exactly reproduces the FHDP result. For that purpose, let us first introduce the following lemma.

\begin{lemma}[Relationship of Taylor and Fourier--Hermite series]
\label{lem:g_trans}
Let us consider a scalar function $g(y)$ and define the following Weierstrass type of transform:
\begin{equation}
  \bar{g}(y) = \int g(z) \, \mathcal{N}(z \mid y, \Sigma) \, dz.
\label{eq:g_transform}
\end{equation}
Let us also assume that $\bar{g}(y)$ is at least three times continuously differentiable, which can be ensured by, for example, $\int \|z\|^3 \, |g(z)| \, \mathcal{N}(z \mid y, \Sigma) \, dz < \infty$.
Then the Taylor series expansion of $\bar{g}(y')$ matches the Fourier--Hermite expansion of $g(y')$ with respect to $\mathcal{N}(z \mid y, \Sigma)$ up to an additive constant.
\end{lemma}

\begin{proof}
When $z \sim \mathcal{N}(z \mid y,\Sigma)$, integration by parts gives
\begin{equation}
\begin{split}
\mathbb{E}[g(z) H_1(\Lambda^{-1}(z-y))] 
&= \Lambda^\top \, \mathbb{E}[G_z(z)], \\
\mathbb{E}[g(z) H_2(\Lambda^{-1}(z-y))]  
&= \Lambda^\top \, \mathbb{E}[G_{zz}(z)] \, \Lambda,
\end{split}
\end{equation}
where $\Sigma = \Lambda \, \Lambda^\top$. Substituting to \eqref{eq:fh_general_second_order_y2} then gives the following Fourier--Hermite series for $g(y)$ with respect to $\mathcal{N}(z \mid y,\Sigma)$:
\begin{equation}
\begin{split}
g(y') &\approx \mathbb{E}[g(z)] + \mathbb{E}[G_z(z)]^\top (y' - y) \\
&\quad + \frac{1}{2}  (y' - y)^\top \mathbb{E}[G_{zz}(z)] \, (y' - y) - \frac{1}{2} \mathop{\text{tr}}\left\{  \mathbb{E}[G_{zz}(z)] \, \Sigma \right\},
\end{split}
\label{eq:fh_version_of_g}
\end{equation}
where the expectations are over $\mathcal{N}(z \mid y,\Sigma)$.

For the Taylor series expansion of \eqref{eq:g_transform} we can change variables by $\int g(z) \, \mathcal{N}(z \mid y, \Sigma) \, dz = \int g(y + \xi) \, \mathcal{N}(\xi \mid 0, \Sigma) \, d\xi$, which after differentiation under integral and changing back to $z$ gives
\begin{equation}
\begin{split}
  \bar{G}_y(y) &= \int G_z(z) \, \mathcal{N}(z \mid y, \Sigma) \, dz, \\
  \bar{G}_{yy}(y) &= \int G_{zz}(z) \, \mathcal{N}(z \mid y, \Sigma) \, dz.
\end{split}
\end{equation}
In the notation of \eqref{eq:fh_version_of_g} we have $\bar{g}(y) = \mathbb{E}[g(z)]$, $
\bar{G}_y(y) = \mathbb{E}[G_z(z)]$, and $\bar{G}_{yy}(y) = \mathbb{E}[G_{zz}(z)]$, and hence the Taylor series expansion of $\bar{g}(y')$ becomes
\begin{equation}
\begin{split}
\bar{g}(y') &\approx \mathbb{E}[g(z)] + \mathbb{E}[G_z(z)]^\top (y' - y) \\
&\quad + \frac{1}{2}  (y' - y)^\top \mathbb{E}[G_{zz}(z)] \, (y' - y),
\end{split}
\label{eq:taylor_version_of_g}
\end{equation}
which is the same as \eqref{eq:fh_version_of_g} except for the last term which is constant in $y'$.
\end{proof}

\begin{lemma}[Equivalent DDP model]
\label{lem:equiv_ddp}
Consider a transformation of the problem \eqref{eq:total_cost} and \eqref{eq:dynamics}, where we replace the end-condition with
\begin{equation}
\begin{split}
  \bar{\ell}_T(x_T) &= \int \ell_T(x'_T) \, \mathcal{N}(x'_T \mid x_T, \Sigma_T) \, dx'_T
\end{split}
\label{eq:bar_ell_T}
\end{equation}
and the cost at time step $k$ by 
\begin{equation}
\begin{split}
  \bar{\ell}_k(x_k,u_k) &= \bar{Q}_k(x_k, u_k)  - V_{k+1}(f_k(x_{k}, u_k)),
\end{split}
\label{eq:bar_ell_k}
\end{equation}
where the transformed action-value function is defined by
\begin{equation}
\begin{split}
  \bar{Q}_k(x_k, u_k) 
  &= \iint \left[ \ell_k(x_k', u_k') + V_{k+1}(f_k(x_{k}', u_k'))
\right] \\
  &\qquad \times \mathcal{N}([x_k';u_k'] \mid [x_k,u_k], \Sigma_k) \, du'_k \, dx'_k.
\end{split}
\label{eq:bar_Q}
\end{equation}
Then the iteration produced by applying DDP to the modified model exactly matches the iterations generated by FHDP.
\end{lemma}

\begin{proof}
Substituting the modified cost function \eqref{eq:bar_ell_k} to \eqref{eq:q_function} reduces the action value function to \eqref{eq:bar_Q}. Then by Lemma~\ref{lem:g_trans}, the second-order Taylor series expansions of the action-value function taken around $\hat{x}_T$ and $\hat{x}_k,\hat{u}_k$ match the Fourier--Hermite series expansions up to a constant which only depends on the nominal trajectory. The maxima of the Taylor-series based action-value functions with respect to the $u_k$ also match the maxima obtained from the Fourier--Hermite series expansions (up to the constant). Therefore the control laws are the same and as the forward passes are the same, the iterations produce exactly the same result.
\end{proof}

\begin{theorem}[Quadratic convergence of FHDP] \label{the:fhdp}
Assume that $f_k$ is three times continuously differentiable and the transformed problem defined in Lemma~\ref{lem:equiv_ddp} has a unique solution $(\bar{x}^*_{1:T},\bar{u}^*_{1:T-1})$ within a convex set $D$. Further assume that the second derivatives of the transformed action-value function in \eqref{eq:bar_Q} with respect to $u_k$ are positive definite and all the iterates produced by FHDP $(\bar{x}^{(i)}_{1:T},\bar{u}^{(i)}_{1:T-1}) \in D$. Then the sequence of iterates $\bar{u}_{1:T-1}^{(i)}$ produced by FHDP converges quadratically to the solution in the sense that where exist $\bar{c} > 0$ such that
\begin{equation}
  \| \bar{u}_k^{(i+1)} - \bar{u}^*_k \| \le \bar{c} \, \| \bar{u}^{(i)} - \bar{u}^*_k \|^2.
\end{equation}
\end{theorem}

\begin{proof}
By Lemma~\ref{lem:equiv_ddp}, FHDP can be seen as DDP which finds the optimum of a transformed problem defined by \eqref{eq:bar_ell_T}, \eqref{eq:bar_ell_k}, and \eqref{eq:bar_Q}. The assumptions ensure that assumptions for Lemma~\ref{lem:ddp_convergence} are satisfied which leads to the result.
\end{proof}

\begin{remark}
  Because the transformed model reduces to the original model when $\Sigma_T,\Sigma_k \to 0$, in this limit, the results of FHDP and DDP match.
\end{remark}

\begin{remark}
In Theorem~\ref{the:fhdp} we had to assume that the dynamics are three times differentiable to adapt the existing DDP convergence results to the current setting. However, as the Fourier--Hermite expansion is always formed for $Q_k$, the convergence result should apply even in the case that the dynamics are not three times differentiable as long as the transformed $\bar{Q}_k$ is smooth, which it in very general conditions is.
\end{remark}

\section{Experiments} 
\label{sec:experiments}
To demonstrate the performance of our proposed method, we consider the classical pendulum and cart-pole models and compare to the classical DDP and the unscented DP (UDP) in~\cite{manchester2016derivative}. All the algorithms were implemented in MATLAB\textsuperscript{\textregistered} 2021b using its automatic differentiation functionality. All experiments were carried on a CPU using AMD EPYC\textsuperscript{\texttrademark} 7643 with  48 Cores and 2.3GHz.\footnote{The source code is available at \url{https://github.com/EEA-sensors/FourierHermiteDynamicProgramming.git}}

\subsection{Pendulum swinging experiment}
\label{sec:pendulum}

First, we consider a pendulum swing-up problem, which was also used in \cite{manchester2016derivative}. The goal is to swing the pendulum from downward position ($\theta = 0$) to upward ($\theta = \pi$) position by using an input torque,$u$, as control. We define the state of the pendulum as $x = [\theta, \dot{\theta}]^\top$, and use a quadratic cost function of the form
\begin{equation}
    \begin{split}
        J &= \frac{1}{2} (x_T - x_{g})^\top {W}_T (x_T - x_{g}) \\
        &\quad + \sum_{k=1}^{T-1} \left\{ \frac{1}{2} (x_k - x_{g})^\top {W}  (x_k - x_{g}) + u_k^\top R u_k \right\}.
    \end{split}
    \label{eq:cost_func}
\end{equation}
The parameters of the pendulum model and the cost function are same as in~\cite{manchester2016derivative}. We discretize the dynamics using a fourth-order Runge--Kutta method with a zero-order hold on $u$. The step size is set to $0.1$ and $T=50$. We set the initial and final states to be $x_1 = [0, 0]^\top$ and $x_{g} = [\pi, 0]^\top$, respectively. 
\begin{figure}[tbh]
     \centering
     \begin{subfigure}[t]{0.49\linewidth}
         \centering
         \includegraphics[width=\linewidth]{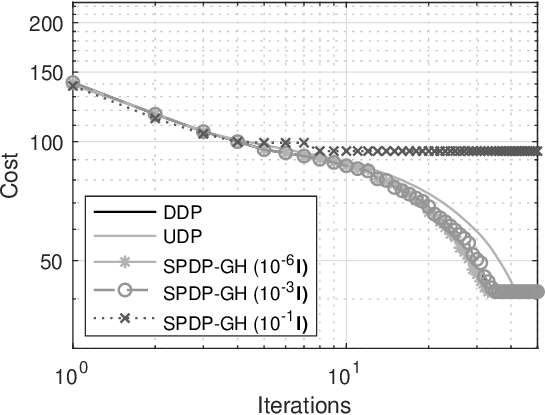}
         \caption{}
         \label{fig:cost_cov_pendulum}
     \end{subfigure}
     \hfill
     \begin{subfigure}[t]{0.49\linewidth}
         \centering
         \includegraphics[width=\linewidth]{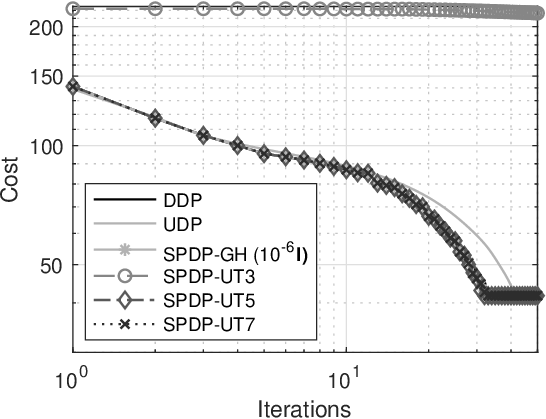}
         \caption{}
         \label{fig:cost_gh_ut_spdp_pendulum}
     \end{subfigure}
     \caption{Reduction in cost for pendulum swing-up problem by using DDP, unscented DP (UDP), and sigma-point based FHDP (SPDP). Results of SPDP in (a) are with Gauss--Hermite rule with $p=3$ (SPDP-GH) and different values of $\Sigma_T$ and $\Sigma_k$. In (b), the results of SPDP are with different integration rules: SPDP-GH is with the Gauss--Hermite ($p=3$), $3^{rd}$ (SPDP-UT3), $5^{th}$ (SPDP-UT5), and $7^{th}$ (SPDP-UT7) order cubature/unscented rules.}
     \label{fig:performance_pendulum}
\end{figure}

Figs.~\ref{fig:cost_cov_pendulum} and \ref{fig:cost_gh_ut_spdp_pendulum} show the the total cost of the trajectory as a function of the iteration number with DDP, UDP, and sigma-point based FHDP (SPDP) methods. As the aim is to compare the performance of DDP, UDP, and different variations of SPDP methods, in Fig.~\ref{fig:cost_cov_pendulum}, we use SPDP with Gauss--Hermite quadrature rule of order $p = 3$ (we call this SPDP-GH) and set $\Sigma_k \in \{10^{-6} I, 10^{-3} I, 10^{-1} I \}$, where $I$ is the identity matrix and $\Sigma_T$ similarly. As can be seen from the figure, all the compared methods except for one SPDP converge to a very similar total cost. In the first few iterations, all methods have approximately similar total cost. In later iterations, however, SPDP-GH with the large covariance, say $\Sigma_T = 10^{-1}I$ and $\Sigma_k = 10^{-1}I$, is slower to reduce the cost and hence requires more iterations to converge. This is expected because a large value for the covariance corresponds to FH expansion which averages the function over a large area around the nominal point. On the other hand, with a small covariance, say $\Sigma_T = 10^{-1}I$, and $\Sigma_k = 10^{-1}I$, the SPDP method coincides with DDP, which confirms the theoretical analysis of the method in Section~\ref{sec:convergence}. It can be seen that in this experiment both DDP and SPDP have better convergence speed than UDP. SPDP-GH with $\Sigma_T = 10^{-1}I$, and $\Sigma_k = 10^{-1}I$ has a slightly better cost reduction compared to DDP (see SPDP-GH ($10^{-6}I$) curve after $30$ iterations).

Fig.~\ref{fig:cost_gh_ut_spdp_pendulum} shows the performance of SPDP method with different sigma-point schemes. The schemes are Gauss--Hermite quadrature rule with $p=3$, $\Sigma_T = 10^{-1}I$, and $\Sigma_k = 10^{-1}I$ (SPDP-GH), 3rd (SPDP-UT3), 5th (SPDP-UT5), and 7th order (SPDP-UT7) unscented transforms, that is, spherical cubature rules~\cite{kokkala2015sigma}. We can see that the 3rd order cubature/unscented rule is not sufficient for computing the integrals for Fourier--Hermite coefficients as discussed in Section~\ref{sec:sigma-points} (see the curve of SPDP-UT3 in Fig.~\ref{fig:cost_gh_ut_spdp_pendulum}). The SPDP-GH, SPDP-UT5, and SPDP-UT7 methods converge with approximately similar number of iterations as DDP, and the performance is practically independent of the integration rule used.

\begin{table}[tbh]
\caption{Average run times and the number of sigma points of DDP, UDP, and different variations of SPDP methods in pendulum swing-up problem. Here, $m_T$, $m_k$ denote the number of sigma-points in SPDP methods at terminal and $k^{th}$ step.}
\setlength{\tabcolsep}{3pt}
\centering
\begin{tabular}{|p{90pt}|c|c|}
\hline
Method &
Average run times (s) & \# of sigma points \\
\hline
 DDP & 2.3775 & -\\ 
 UDP & 0.2576 & 6\\  
 SPDP-GH ($p=3$) & 0.0113 & $m_T = 9$, $m_k = 27$\\
 SPDP-UT5 & 0.0071 & $m_T = 9$, $m_k = 19$\\
 SPDP-UT7 & 0.0140 & $m_T = 17$, $m_k = 45$ \\
\hline
\end{tabular}
\label{tab:pendulum_avg_time}
\end{table}

Table~\ref{tab:pendulum_avg_time} lists the average run times (in s) to compute backward and forward passes per iteration. As we can see, DDP requires more computational time due to computing the derivatives appearing in~\eqref{eq:coefficients_of_q}. UDP requires less time since the method avoids computing derivatives of $f_k$. However, it requires computing the derivatives of $l_k$ and backward propagation of sigma points. The computational speed of SPDP mainly depends on the number of sigma-points used in the integration rule. We also list the number of sigma-points that need to be evaluated for UDP and SPDP methods in Table~\ref{tab:pendulum_avg_time}. The number of sigma points at terminal step T and at step k are denoted as $m_T$ and $m_k$, for $k = 1, \ldots, T-1$. It is clear from this table that SPDP is faster than the other methods. The run times for SPDP-UT5 is the fastest among all the methods,  because the number of evaluation points in SPDP-UT5 is the least of the SPDP methods. 

\subsection{Cart-pole experiment}
\label{sec:cartpole_experiment}
In this experiment, we consider a cart-pole balancing problem, where the aim is to balance the pole in upward position by applying an external force $u$ to move the cart in the horizontal direction. The similar experiment was also performed in~\cite{manchester2016derivative}. The cart with mass $m_c$ is attached to a pole with mass $m_p$ and length $l$. We denote the state of this system as $x = [v, \theta, \dot{v}, \dot{\theta}]^\top$, where $v$ and $\dot{v}$ are the position and the velocity of the cart, respectively, and $\theta$ and $\dot{\theta}$ denote the angle and angular speed of the pole, respectively. We set the initial and final states to be $x_1 = [0, 0, 0, 0]^\top$ and $x_{g} = [0, \pi, 0, 0]^\top$, respectively. The differential equations of this cart-pole system can be found in~\cite{barto1983neuronlike}. We discretize the dynamics using fourth-order Runge--Kutta integration and zero-order hold for the control $u$.

We use a cost function of similar form as \eqref{eq:cost_func} and set the values $m_p$, $m_c$, $l$, $g$, $W_T$, $W$, and $R$ to be the same as in~\cite{manchester2016derivative}. The step size is set to be $0.1$ and $T = 50$. 

\begin{figure}[tbh]
     \centering
     \begin{subfigure}[t]{0.49\linewidth}
         \centering
         \includegraphics[width=\linewidth]{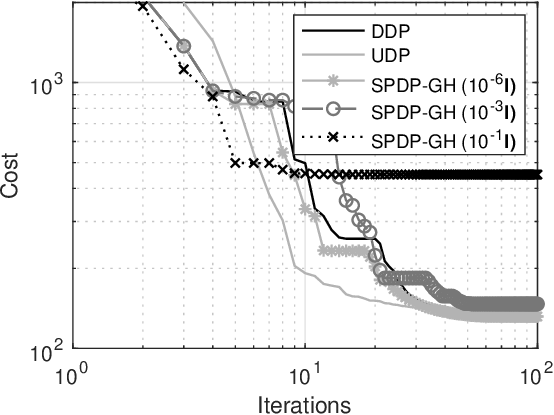}
         \caption{}
         \label{fig:cost_cov_cartpole}
     \end{subfigure}
     \hfill
     \begin{subfigure}[t]{0.49\linewidth}
         \centering
         \includegraphics[width=\linewidth]{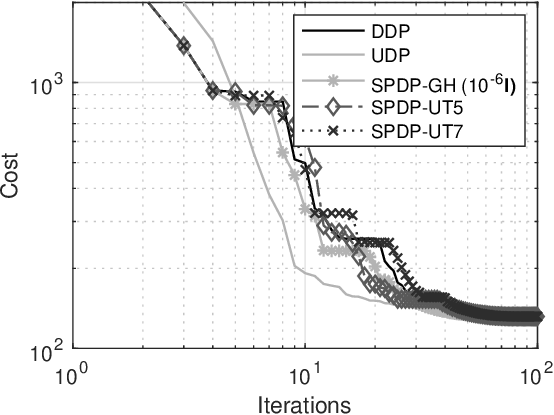}
         \caption{}
         \label{fig:cost_gh_ut_spdp_cartpole}
     \end{subfigure}
     \caption{Reduction in cost for cart-pole balancing problem by using DDP, unscented DP (UDP) and sigma-point based FHDP (SPDP). Results in (a) are with SPDP with Gauss--Hermite rule of order $p=3$ (SPDP-GH) and different values of $\Sigma_T$ and $\Sigma_k$, and in (b) for SPDP with different integration rules and $10^{-6}I$ covariance: GH with $p=3$ (SPDP-GH),  $5^{th}$ (SPDP-UT5), and $7^{th}$ (SPDP-UT7) order cubature/unscented rules.}
     \label{fig:performance_cartpole}
\end{figure}

Similar to pendulum example, we investigate the performance of the methods in reducing total cost. The results of SPDP method with different covariances and different integration schemes are shown in Figs.~\ref{fig:cost_cov_cartpole} and~\ref{fig:cost_gh_ut_spdp_cartpole}. In this case the covariance of SPDP affects the behavior more than in the pendulum experiment. For the first few iterations in Fig.~\ref{fig:cost_cov_cartpole}, all methods have a fast cost reduction, fastest being SPDP with $10^{-1} I$ covariance. For the later iterations, the methods have different speeds of cost reduction. With larger covariances (see SPDP-GH ($10^{-1}I$) and SPDP-GH ($10^{-3}I$) in Fig.~\ref{fig:cost_cov_cartpole}), SPDP method is slower in reducing the total cost and does not reach convergence within the 100 iterations shown in the figure. With smaller covariances, SPDP has similar behavior as DDP. What is interesting in this figure is that SPDP-GH ($10^{-6} I$) has better cost reduction compared to DDP during intermediate iterations. The SPDP-GH ($10^{-1} I$) has the fastest cost reduction until the first $6$ iterations. After that, there is no improvement. We also observed that the convergence of UDP method was the fastest among all the methods. In Fig.~\ref{fig:cost_gh_ut_spdp_cartpole} we can see that the integration method has a slight effect on the performance, but the results of SPDPs with different integration rules are practically the same.

\begin{table}[tbh]
\caption{Average run times and the number of sigma points of DDP, UDP, and different variations of SPDP methods in cart-pole balancing problem. Here, $m_T$, $m_k$ denote the number of sigma points in SPDP methods at terminal and $k^{th}$ step.}
\setlength{\tabcolsep}{3pt}
\centering
\begin{tabular}{|p{85pt}|c|c|}
\hline
Method &
Average run times (s) & \# of sigma points\\ 
\hline
 DDP & 34.8622 & - \\ 
 UDP & 0.3222 & 10\\  
 SPDP-GH ($p=3$) & 0.0999 & $m_T = 81$, $m_k = 243$ \\
 SPDP-UT5 &  0.0239 & $m_T = 33$, $m_k = 51$ \\
 SPDP-UT7 & 0.0758 & $m_T = 97$, $m_k = 181$ \\
\hline
\end{tabular}
\label{tab:cartpole_avg_time}
\end{table}

The average run times per iteration to compute the backward pass and forward passes are listed in Table~\ref{tab:cartpole_avg_time}. As in the pendulum case, the SPDP methods are faster than the other methods, UT5-based method being fastest of them. However, the margin to the UDP method is now smaller as UDP requires a relatively smaller number of sigma points than SPDP methods.

\subsection{Quadcopter experiment}
\label{sec:quadcopter}
Finally, we consider a multi-rotor unmanned quadcopter, which has 4 rotors with 6 degrees of freedom \cite{ahmed2018sliding}. The state of the system contains the 3D coordinates, velocities, the orientation (roll, pitch, yaw), and the angular velocities. There are 4 control inputs consisting of the total thrust produced by 4 rotors and the input torques. The cost function is similar to \eqref{eq:cost_func} and 4-th order Runge--Kutta method is used for discretization. Table~\ref{tab:quadcopter_avg_time} shows the results with different methods. DDP method is the slowest among all while SPDP-UT5 and SPDP-UT7 showed competitive results. The SPDP-GH method is not feasible since the number of evaluation points increases exponentially as the number of dimensions increases.

\begin{table}[tbh]
\caption{Average run times and the number of sigma points of DDP, and different variations of SPDP methods in quadcopter problem. Here, $m_T$, $m_k$ denote the number of sigma points in SPDP methods at terminal and $k^{th}$ step.}
\setlength{\tabcolsep}{3pt}
\centering
\begin{tabular}{|p{65pt}|c|c|}
\hline
Method &
Average run times (s) & \# of sigma points\\ 
\hline
 DDP & 480.8027 & - \\ 
 SPDP-UT5 &  0.1520 & $m_T = 289$, $m_k = 513$ \\
 SPDP-UT7 & 1.5834 & $m_T = 2,337$, $m_k = 5,505$ \\
\hline
\end{tabular}
\label{tab:quadcopter_avg_time}
\end{table}

\section{Conclusion and Discussion}
\label{sec:conclusion}
In this paper, we have proposed a Fourier--Hermite series based Fourier--Hermite dynamic programming (FHDP) algorithm and its derivative-free implementation sigma-point dynamic programming (SPDP) that approximates the action-value function using Fourier--Hermite series and sigma points. This is in contrast to classical differential dynamic programming (DDP) which is based on the use of Taylor series expansion. This new SPDP has the performance close or better than DDP algorithm and while it is computationally faster as the high order derivatives do not need to be evaluated. As shown by the experiments, it also can outperform another sigma-point based DP method, unscented dynamic programming (UDP), and it is also computationally faster in the tested experiments. We have also proved the local second order convergence of the proposed method.

Although in the second experiment, UDP outperformed the proposed method, in Fig.~\ref{fig:cost_cov_cartpole}, we see that SPDP-GH with different $\Sigma_T, \Sigma_k$ produces different convergence behavior. Moreover, we notice a larger cost reduction in SPDP-GH $(10^{-1}I)$ than UDP until $6^{th}$ iteration. This indicates that the covariance schedule of SPDP could be used to further improve the convergence speed of the method, and this is also confirmed by additional numerical experiments that we have done. However, we leave the further investigation of the covariance schedule as a future work.

\appendices


\bibliographystyle{IEEEtran}
\bibliography{ocbib}

\end{document}